\documentclass[11pt,a4paper]{amsart}
\usepackage{amsmath,amssymb,color,multicol,setspace}
\usepackage[utf8,utf8x]{inputenc}

\usepackage[pdftitle = {Frieze\ patterns\ over\ integers\ and\ other\ subsets\ of\ the\ complex\ numbers}]{hyperref}
\usepackage{longtable}
\usepackage{xy,amscd}
\usepackage{epsfig}
\usepackage{rotating}
\usepackage{xspace}
\usepackage{enumitem}
\xyoption{all}
\usepackage{tikz}
\usetikzlibrary{arrows,decorations.pathmorphing,decorations.pathreplacing,positioning,shapes.geometric,shapes.misc,decorations.markings,decorations.fractals,calc,patterns}
\usepackage{xcolor}
\usepackage{comment}

\newcommand{\blue}{\color{blue}}
\newcommand{\red}{\color{red}}

\usepackage{fullpage}

\newtheorem{Lemma}{Lemma}[section]
\newtheorem{Theorem}[Lemma]{Theorem}
\newtheorem{Proposition}[Lemma]{Proposition}
\newtheorem{Corollary}[Lemma]{Corollary}

\theoremstyle{definition}
\newtheorem{Definition}[Lemma]{Definition}

\newtheorem{Remark}[Lemma]{Remark}

\newtheorem{Example}[Lemma]{Example}

\numberwithin{equation}{section}

\newcommand{\sps}{\makebox[20pt]{\ }}

\title{Noncommutative frieze patterns with coefficients}

\author{Michael~Cuntz}
\address{Michael Cuntz, Leibniz Universit\"at Hannover,
Institut f\"ur Algebra, Zahlentheorie und Dis\-krete Mathematik,
Fakult\"at f\"ur Mathematik und Physik,
Welfengarten 1,
D-30167 Hannover, Germany}
\email{cuntz@math.uni-hannover.de}
\urladdr{https://www.iazd.uni-hannover.de/de/cuntz}

\author{Thorsten~Holm}
\address{Thorsten Holm, Leibniz Universit\"at Hannover,
Institut f\"ur Algebra, Zahlentheorie und Dis\-krete Mathematik,
Fakult\"at f\"ur Mathematik und Physik,
Welfengarten 1,
D-30167 Hannover, Germany}
\email{holm@math.uni-hannover.de}
\urladdr{https://www.iazd.uni-hannover.de/de/holm}

\author{Peter J{\o}rgensen}
\address{Peter J{\o}rgensen, 
Department of Mathematics,
Aarhus University,
Ny Munkegade 118,
8000 Aarhus C,
Denmark}
\email{peter.jorgensen@math.au.dk}
\urladdr{https://sites.google.com/view/peterjorgensen}



\begin{document}

\begin{abstract}
Based on Berenstein and Retakh's notion of noncommutative polygons
\cite{BR18}  
we introduce and study noncommutative frieze patterns. 
We generalize several notions and fundamental properties from the classic 
(commutative) frieze patterns to noncommutative frieze patterns,
e.g.\ propagation formulae and $\mu$-matrices, quiddity cycles and reduction formulae,
and we 
show that local noncommutative exchange relations and local triangle relations 
imply all noncommutative exchange relations and triangle relations. Throughout, we allow coefficients, so 
we obtain generalizations of results from our earlier paper \cite{CHJ20}
from the commutative to the noncommutative setting. 
\end{abstract}

\maketitle

\section{Introduction}

Frieze patterns are certain arrays of numbers introduced by Coxeter \cite{Cox71} and studied further by Conway and Coxeter \cite{CC73} in the 1970's. 
They are closely linked to Fomin and Zelevinsky's cluster algebras and hence form a currently very active research area connecting topics like combinatorics,
geometry, number theory, representation theory and integrable systems.

In their paper {\em Noncommutative marked surfaces} \cite{BR18}, Berenstein and Retakh introduce 
to each marked surface $\Sigma$ a noncommutative algebra $A_{\Sigma}$ generated by 
noncommutative geodesics between marked points. The relations in this algebra are certain
triangle relations and noncommutative versions of Ptolemy relations. Berenstein and Retakh 
show that these algebras have similar properties as Fomin and Zelevinsky's
cluster algebras, e.g.\ there is a noncommutative Laurent phenomenon in $A_{\Sigma}$ for
any triangulation of the surface $\Sigma$. 
See also the paper \cite{GK21} by Goncharov and Kontsevich for the construction of a large class of noncommutative 
cluster algebras. 

In this paper we show that Berenstein and Retakh's notion of noncommutative polygons 
can be used to define noncommutative frieze patterns and to develop a general theory of these. 
It is well-known that (commutative) frieze patterns of height $n$ can be seen as friezes
on $(n+3)$-gons, i.e.\ as maps assigning values to each (undirected) diagonal of the polygon such that 
for every pair of crossing diagonals the (commutative) Ptolemy relation is satisfied.

In the noncommutative setting, a crucial difference in the combinatorial model
is that for each pair of vertices 
of the polygon there are two directed diagonals to which values are assigned. The relations 
suggested by Berenstein and Retakh are noncommutative exchange relations, or alternatively 
noncommutative Ptolemy relations, and, as a truly noncommutative
feature, triangle relations. See Section \ref{sec:ncpr} for details on noncommutative polygons,
including all the relevant definitions from \cite{BR18}. 

In the classic commutative case, frieze patterns are defined as certain arrays of numbers such that
each $2\times 2$-submatrix has determinant 1. This can be seen as certain local relations,
for quadrangles whose vertices are pairs of consecutive vertices of the polygon, and with 
all boundary entries equal to 1. For the noncommutative case, we present a generalization of
such local relations in Section \ref{sec:local}. 
Berenstein and Retakh's noncommutative polygons then naturally lead to a notion of noncommutative 
frieze pattern which we introduce and discuss in Section~\ref{sec:friezepattern}. 
A noncommutative frieze pattern over a ring $R$ is an array of invertible elements of $R$ of the form
$$\begin{array}{ccccccccccccc}
0~~ & c_{0,1} & c_{0,2} & c_{0,3}~~~ & \ldots & c_{0,m-1} & 0 & & & & & &\\
~~ \\
& 0 & c_{1,2} & c_{1,3} & \ldots & c_{1,m-1} &c_{1,0} & 0 & & & & &\\
~~ \\
& & 0 & c_{2,3} & \ldots & c_{2,m-1} & c_{2,0} & c_{2,1} & ~~~~~0 & & & &\\
~~\\
& &  & \ddots  &  \ddots & \vdots & \vdots &  \vdots & \ddots & \ddots & & & \\
~~\\
& & &  & \ddots & c_{m-2,m-1} & c_{m-2,0} & c_{m-2,1} & \ldots  & c_{m-2,m-3} & 0 & & \\
~~\\
& & & &  &   0 & c_{m-1,0} &  c_{m-1,1} & \ldots & \ldots &  
c_{m-1,m-2} & ~~~~~~~~0 &  \\
\end{array}
$$
such that all local triangle relations and all local noncommutative exchange relations are satisfied.
Here is an example of a noncommutative frieze pattern over the quaternions (see Example \ref{ex:friezes}):
$$
\begin{array}{cccccccccccc}
0 & ~1~ & ~~i~~ & 1 - k & -i - 2j & 1 & 0 & \sps & \sps & \sps & \sps & \sps \\
    & 0 & 1 & -2i - j & 3k & -i + j & 1 & 0 &    &    &    &   \\
    &    & 0 & 1 & i - j & k & i & 1 & 0 &    &    &   \\
    &    &    & 0 & 1 & j & 1 + k & -2i - j & 1 & 0 &    &   \\
    &    &    &    & 0 & 1 & -i - 2j & -3k & ~i - j~ & 1 & 0 &   \\
\sps & \sps & \sps & \sps & \sps & 0 & 1 & -i + j & -k & ~j~ & ~1~ & 0
\end{array}
$$
This noncommutative frieze pattern has height 3, i.e.\ it corresponds to a noncommutative frieze on
a hexagon. Note that for all values on diagonals of length 3 we have $c_{i,i+3}\neq c_{i+3,i}$, that is, 
there is no glide symmetry (which always holds for classic commutative frieze patterns).

For classic frieze patterns, the notion of tameness is fundamental, as introduced by
Bergeron and Reutenauer \cite{BR10}. Large parts of the classic 
theory only work for tame frieze patterns. To transfer results to noncommutative frieze patterns
we need a suitable replacement. Recall that a frieze pattern is called tame if each
neighbouring $3\times 3$-submatrix has determinant 0. 

In the noncommutative case, determinants can be 
replaced by certain quasideterminants. 
For $2\times 2$-matrices this notion of quasideterminant goes back to Richardson \cite{R26} and Heyting \cite{H28}. A general theory and many fundamental properties of quasideterminants have been
established by Gelfand and Retakh \cite{GR91},\cite{GR92}.
In Section \ref{sec:quasidet} we briefly collect some of the few fundamental properties of
quasideterminants we shall need. 
We define a noncommutative frieze pattern to be tame if for 
each neighbouring $3\times 3$-submatrix $M$ the quasideterminant $|M|_{3,3}$ is equal to zero.
Later in Section \ref{sec:theoryncfriezes} we show that every noncommutative frieze pattern is tame 
in this sense. This can be seen as a noncommutative analogue of the well-known fact that 
every classic frieze pattern without zero entry is tame (see \cite[Proposition 1]{BR10}).

The main parts of this paper are Sections \ref{sec:theoryncfriezes} and \ref{sec:localimplyall} 
in which we develop a theory of noncommutative frieze patterns. This generalizes several classic results
about frieze patterns from the commutative to the noncommutative setting. As a first main result
we prove in Theorem \ref{prop:propagation} propagation formulae for the entries in a noncommutative 
frieze pattern, generalizing the result \cite[Proposition 2.10]{CHJ20} for commutative frieze patterns
with coefficients. We then give various applications of Theorem \ref{prop:propagation}, including 
the notion of noncommutative $\mu$-matrices. These are certain $2\times 2$-matrices which in the 
commutative case, are the key ingredients for the classic Conway-Coxeter theory and beyond.  
We also define noncommutative quiddity cycles and prove a reduction formula for these, see 
Proposition \ref{prop:qcreduce}. This generalizes the classic formula in Conway-Coxeter theory 
for removing/inserting 1's in the quiddity cycle, which underlies the bijection between frieze patterns 
over positive integers and triangulations of polygons \cite{CC73}.

In Section \ref{sec:localimplyall} we consider again the relations defining noncommutative frieze patterns,
namely local triangle relations and local noncommutative exchange relations (see Definition 
\ref{def:nclocal}).
Our main result 
in this section is the following (restated as Theorem \ref{thm:localimplyall}):
\medskip

\noindent
{\bf Theorem.}
{\em Let $c:\mathrm{diag}(\mathcal{P})\to R^{\ast}$ be a noncommutative frieze on a polygon $\mathcal{P}$
over a ring $R$. Then $c$ satisfies all triangle relations and all noncommutative exchange relations.}
\medskip

In other words, in a noncommutative frieze pattern, the entries satisfy many more relations than just
the local relations defining them. This gives a noncommutative analogue of the analogous well-known 
fact for classic frieze patterns.

\section{Noncommutative polygons} \label{sec:ncpr}

Let $\mathcal{P}$ be an $m$-gon (where $m\ge 3$), with vertices numbered consecutively
$0,1,\ldots,m-1$. For each pair of vertices $i,j\in\{0,1,\ldots,m-1\}$ we have two directed
diagonals, one from $i$ to $j$ and one from $j$ to $i$, see Figure \ref{fig:diag}. Note that also for consecutive 
vertices $i$ and $i+1$ we have such directed diagonals on the boundary of $\mathcal{P}$.  
Let $\mathrm{diag}(\mathcal{P})$ denote
the set of all directed diagonals in $\mathcal{P}$. 

\begin{figure} 
\begin{center}
\begin{tikzpicture}[auto]
    \node[name=s, draw, shape=regular polygon, regular polygon sides=500, minimum size=3cm] {};
    \draw[thick, -stealth] (s.corner 457) to (s.corner 200);
    \draw[thick, stealth-] (s.corner 450) to (s.corner 207);
    
    \draw[shift=(s.corner 198)]  node[below]  {{\small $j$}};
    \draw[shift=(s.corner 448)]  node[above]  {{\small $i$}};
   \end{tikzpicture}  
\end{center}
\caption{The directed diagonals in the polygon $\mathcal{P}$.}
\label{fig:diag}
\end{figure}
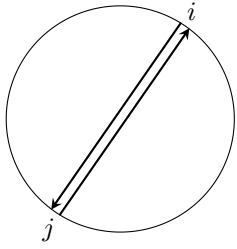

The following definition is fundamental for this paper. 

\begin{Definition}[Berenstein, Retakh \cite{BR18}]
\label{def:ncpolygon}
Let $R$ be a ring and 
$\mathcal{P}$ be an $m$-gon (where $m\ge 3$). Suppose that to each pair of vertices $i,j$ of
$\mathcal{P}$ invertible elements $c_{i,j}\in R^{\ast}$ and $c_{j,i}\in R^{\ast}$ are assigned, that is, we have a map 
$f:\mathrm{diag(\mathcal{P})}\to R^{\ast}$.
\begin{enumerate}
\item[{(a)}] The {\em triangle relation} for distinct vertices $i,j,k$ of $\mathcal{P}$ is the equation
\begin{equation} \label{eq:triangle}
c_{i,j}c_{k,j}^{-1}c_{k,i} = c_{i,k}c_{j,k}^{-1}c_{j,i}.
\end{equation}
See Figure \ref{fig:triangle}.
\item[{(b)}] Consider a quadrangle with vertices $i,j,k,\ell$ in $\mathcal{P}$ such that $(i,k)$ and $(j,\ell)$
are the diagonals of the quadrangle.
The {\em noncommutative exchange relation} for the diagonal $(k,i)$ is the equation
\begin{equation} \label{eq:ncexchange}
c_{k,i} = c_{k,\ell}c_{j,\ell}^{-1}c_{j,i}+ c_{k,j} c_{\ell,j}^{-1}c_{\ell,i}
\end{equation}
See Figure \ref{fig:exchange}.
\end{enumerate}
\end{Definition}

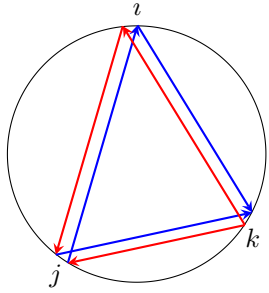
\begin{figure}
\begin{center}
\begin{tikzpicture}[auto]
    \node[name=s, draw, shape=regular polygon, regular polygon sides=500, minimum size=3.4cm] {};
    \draw[blue,thick, stealth-] (s.corner 1) to (s.corner 207);
    \draw[red,thick, -stealth] (s.corner 10) to (s.corner 198);
    \draw[blue,thick, -stealth] (s.corner 198) to (s.corner 339);
    \draw[red,thick, stealth-] (s.corner 207) to (s.corner 330);
    \draw[blue,thick, -stealth] (s.corner 1) to (s.corner 339);
    \draw[red,thick, stealth-] (s.corner 10) to (s.corner 330);
    
    \draw[shift=(s.corner 198)]  node[below]  {{\small $j$}};
    \draw[shift=(s.corner 1)]  node[above]  {{\small $i$}};
    \draw[shift=(s.corner 320)]  node[right]  {{\small $k$}};
   \end{tikzpicture}  
\end{center}
\caption{The triangle relation ${\red c_{i,j}c_{k,j}^{-1}c_{k,i}} = {\blue c_{i,k}c_{j,k}^{-1}c_{j,i}}$.
}
 \label{fig:triangle}
\end{figure}

\begin{figure} 
\begin{center}
\begin{tikzpicture}[auto]
    \node[name=s, draw, shape=regular polygon, regular polygon sides=500, minimum size=3.4cm] {};
    \draw[blue,thick, stealth-] (s.corner 53) to (s.corner 182);
    \draw[red,thick, -stealth] (s.corner 187) to (s.corner 306);
    \draw[blue,thick, stealth-] (s.corner 50) to (s.corner 316);
    \draw[red,thick, -stealth] (s.corner 45) to (s.corner 310);
    \draw[thick, -stealth] (s.corner 185) to (s.corner 410);
    \draw[red,thick, -stealth] (s.corner 45) to (s.corner 413);
     \draw[blue,thick, -stealth] (s.corner 318) to (s.corner 406);
    
    \draw[shift=(s.corner 178)]  node[below]  {{\small $k$}};
    \draw[shift=(s.corner 50)]  node[above]  {{\small $j$}};
    \draw[shift=(s.corner 300)]  node[right]  {{\small $\ell$}};
    \draw[shift=(s.corner 420)]  node[right]  {{\small $i$}};
   \end{tikzpicture}  
\end{center}
\caption{The noncommutative exchange relation
$c_{k,i} = {\red c_{k,\ell}c_{j,\ell}^{-1}c_{j,i}} + {\blue c_{k,j} c_{\ell,j}^{-1}c_{\ell,i}}$.
}
\label{fig:exchange}
\end{figure}
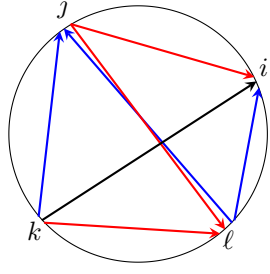

\begin{Remark} \label{rem:relations}
\begin{enumerate}
\item[{(1)}] The triangle relations (\ref{eq:triangle}) can be inverted and then take the form
$$c_{k,i}^{-1} c_{k,j}c_{i,j}^{-1} = c_{j,i}^{-1} c_{j,k} c_{i,k}^{-1}.
$$
\item[{(2)}] The triangle relation as in Figure \ref{fig:triangle} seems to single out the starting vertex $i$.
However, the other triangle relations for this triangle follow from this one. For instance, multiplying
the given triangle relation  $c_{i,j}c_{k,j}^{-1}c_{k,i} = c_{i,k}c_{j,k}^{-1}c_{j,i}$ from the left
by $c_{j,k}c_{i,k}^{-1}$ and from the right by $c_{k,i}^{-^1}c_{k,j}$ yields the triangle relation
$c_{j,k}c_{i,k}^{-1} c_{i,j} = c_{j,i} c_{k,i}^{-1} c_{k,j}$ starting at the vertex $j$.  
\item[{(3)}] 
The noncommutative exchange relation (\ref{eq:ncexchange}) seems to single out the diagonal 
$(k,i)$. However, when the triangle relations (\ref{eq:triangle}) are assumed then one can deduce from 
(\ref{eq:ncexchange}) the analogous equations with the other diagonals singled out.

For instance, multiplying (\ref{eq:ncexchange}) from the left by $c_{k,j}^{-1}$ and using the triangle 
relation for the triangle with vertices $j,k,i$ yields
$$c_{i,j}^{-1}c_{i,k}c_{j,k}^{-1}c_{j,i} = c_{k,j}^{-1}c_{k,\ell}c_{j,\ell}^{-1}c_{j,i} + c_{\ell,j}^{-1}c_{\ell,i}.
$$
Now we multiply from the left by $c_{i,j}$ and from the right by $c_{j,i}^{-1}c_{j,k}$ and get
$$c_{i,k} = c_{i,j}c_{k,j}^{-1}c_{k,\ell}c_{j,\ell}^{-1}c_{j,k} + c_{i,j}c_{\ell,j}^{-1}c_{\ell,i}c_{j,i}^{-1}c_{j,k}.
$$
Finally we use the triangle relations (in the inverted form as in (1)) for the triangles $j,k,\ell$ and $j,\ell,i$ and 
get the desired noncommutative exchange relation
$$c_{i,k} =c_{i,j}c_{\ell,j}^{-1}c_{\ell,k} + c_{i,\ell}c_{j,\ell}^{-1}c_{j,k}
$$
for the directed diagonal $(i,k)$. 

By similar computations one can also obtain the noncommutative exchange relations for the directed
diagonals $(j,\ell)$ and $(\ell,j)$ from (\ref{eq:ncexchange}). 
\item[{(4)}] For a given quadrangle with vertices $i,j,k,\ell$ there are four noncommutative
exchange relations, one for each directed diagonal in the quadrangle. We have observed in (3) that 
any of these noncommutative exchange relations implies the other three, if all triangle relations 
inside the quadrangle are assumed. Without assuming the triangle relations, one has to consider 
the four noncommutative exchange relations separately, as they can not always be deduced from each other. 
\end{enumerate}
\end{Remark}

\section{Local relations} \label{sec:local}

Based on Berenstein and Retakh's
definition of noncommutative relations on a polygon (cf.\ Definition \ref{def:ncpolygon})
we introduce in this section the notion of a noncommutative
frieze on a polygon $\mathcal{P}$ over a ring $R$ as a map $\mathrm{diag}(\mathcal{P})\to R^{\ast}$. Following 
Berenstein and Retakh's approach one could impose 
all triangle relations and all noncommutative exchange relations. Instead we will define a noncommutative
frieze by a much smaller set of relations, so-called local relations. It will then be one of the main results of 
this paper to show
that these local relations imply all triangle and all noncommutative exchange relations. 
We will achieve this goal in Section \ref{sec:localimplyall}. In this 
section we will introduce the relevant local relations and collect some first basic properties, in particular we
will show how to deduce some further relations from the local relations. 

We recall the definition of classic frieze patterns (with coefficients) \cite{Cox71}. These are arrays of numbers 
as in Figure \ref{fig:pattern} (but with $c_{i,j}=c_{j,i}$) such that the determinant of each neighbouring 
$2\times 2$-submatrix is the product of certain boundary entries. 
In the classic theory of tame frieze patterns (with coefficients)
one then shows that 
from these local relations all Ptolemy relations follow (see e.g.\ \cite[Theorem 3.3]{CHJ20}). 
We shall generalize this statement to the noncommutative setting. To this end we have to come up with a suitable
notion of local triangle relations and local noncommutative exchange relations. 
For the exchange relations this is a straightforward 
transfer from the commutative case, but for triangle relations (which do not exist in the commutative case)
this is less obvious. 

\begin{Definition}
\label{def:nclocal}
Let $R$ be a ring and 
$\mathcal{P}$ be an $m$-gon (where $m\ge 3$). Suppose that to each pair of vertices $i,j$ of
$\mathcal{P}$ invertible elements $c_{i,j}\in R^{\ast}$ and $c_{j,i}\in R^{\ast}$ are assigned, that is, we have a map 
$c:\mathrm{diag(\mathcal{P})}\to R^{\ast}$.
\begin{enumerate}
\item[{(a)}] The {\em local triangle relation} for three consecutive vertices $i,i+1+i+2$ of $\mathcal{P}$ 
is the equation
\begin{equation} \label{eq:localtriangle}
c_{i+1,i+2}c_{i,i+2}^{-1}c_{i,i+1} = c_{i+1,i}c_{i+2,i}^{-1}c_{i+2,i+1}.
\end{equation}
See Figure \ref{fig:localtriangle}.
\item[{(b)}] Consider a quadrangle with vertices $i,i+1,j,j+1$ in $\mathcal{P}$.
The {\em local noncommutative exchange relation} for the diagonal $(j,i)$ is the equation
\begin{equation} \label{eq:nclocalexchange}
c_{j,i} = c_{j,j+1}c_{i+1,j+1}^{-1}c_{i+1,i} + c_{j,i+1} c_{j+1,i+1}^{-1}c_{j+1,i}.
\end{equation}
See Figure \ref{fig:localexchange}.
There are analogous local 
noncommutative exchange relations for the other three diagonals of this quadrangle, namely:
$$
c_{i,j} = c_{i,i+1}c_{j+1,i+1}^{-1}c_{j+1,j} + c_{i,j+1} c_{i+1,j+1}^{-1}c_{i+1,j},
$$
$$
c_{i+1,j+1} = c_{i+1,i}c_{j,i}^{-1}c_{j,j+1} + c_{i+1,j} c_{i,j}^{-1}c_{i,j+1},
$$
$$
c_{j+1,i+1} = c_{j+1,j}c_{i,j}^{-1}c_{i,i+1} + c_{j+1,i} c_{j,i}^{-1}c_{j,i+1}.
$$
\item[{(c)}] A {\em noncommutative frieze} over a ring $R$ on a polygon $\mathcal{P}$ is a map $f:\mathrm{diag(\mathcal{P})}\to R^{\ast}$
satisfying all local triangle relations and all local noncommutative exchange relations (for all diagonals in the 
quadrangles with vertices $i,i+1,j,j+1$). 
\end{enumerate}
\end{Definition}

\begin{Remark} \label{rem:nclocal}
\begin{enumerate}
\item[{(1)}] We have observed in Remark \ref{rem:relations}\,(2) that the triangle relations
for a triangle with vertices $i,j,k$ do not depend on the starting point, that is, if one triangle relation for a fixed triangle
holds then all triangle relations for this triangle hold.
\item[{(2)}] For a quadrangle with vertices $i,j,k,\ell$ there are four different exchange relations, one for each directed 
diagonal of the quadrangle, see Definition \ref{def:ncpolygon}.
We have seen in Remark \ref{rem:relations}\,(3) that if one assumes 
all triangle relations inside the quadrangle 
then any of the four exchange relations implies the other three exchange relations for
this quadrangle. However, this is no longer true if one only assumes a selection of triangle relations, e.g.\ if only local triangle relations as in Definition \ref{def:nclocal} are assumed.  
\end{enumerate}
\end{Remark}

\begin{figure}
\begin{center}
\begin{tikzpicture}[auto]
    \node[name=s, draw, shape=regular polygon, regular polygon sides=500, minimum size=3.4cm] {};
    \draw[blue,thick, stealth-] (s.corner 1) to (s.corner 107);
    \draw[red,thick, -stealth] (s.corner 10) to (s.corner 98);
    \draw[blue,thick, -stealth] (s.corner 98) to (s.corner 419);
    \draw[red,thick, stealth-] (s.corner 107) to (s.corner 410);
    \draw[blue,thick, -stealth] (s.corner 1) to (s.corner 419);
    \draw[red,thick, stealth-] (s.corner 10) to (s.corner 410);
    
    \draw[shift=(s.corner 98)]  node[left]  {{\small $i+2$}};
    \draw[shift=(s.corner 1)]  node[above]  {{\small $i+1$}};
    \draw[shift=(s.corner 400)]  node[right]  {{\small $i$}};
   \end{tikzpicture}  
\end{center}
\caption{The local triangle relations
$
{\red c_{i+1,i+2}c_{i,i+2}^{-1}c_{i,i+1}} = {\blue c_{i+1,i}c_{i+2,i}^{-1}c_{i+2,i+1}}
$.}
\label{fig:localtriangle}
\end{figure}
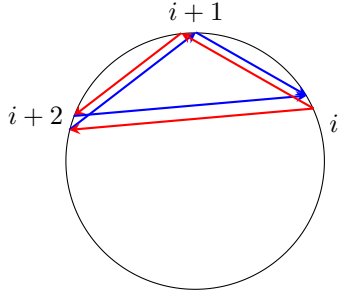

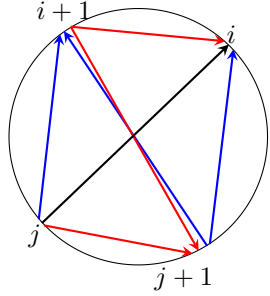
\begin{figure}
\begin{center}
\begin{tikzpicture}[auto]
    \node[name=s, draw, shape=regular polygon, regular polygon sides=500, minimum size=3.4cm] {};
    \draw[blue,thick, stealth-] (s.corner 53) to (s.corner 182);
    \draw[red,thick, -stealth] (s.corner 187) to (s.corner 286);
    \draw[blue,thick, stealth-] (s.corner 50) to (s.corner 296);
    \draw[red,thick, -stealth] (s.corner 45) to (s.corner 290);
    \draw[thick, -stealth] (s.corner 185) to (s.corner 440);
    \draw[red,thick, -stealth] (s.corner 45) to (s.corner 443);
     \draw[blue,thick, -stealth] (s.corner 298) to (s.corner 436);
    
    \draw[shift=(s.corner 178)]  node[below]  {{\small $j$}};
    \draw[shift=(s.corner 50)]  node[above]  {{\small $i+1$}};
    \draw[shift=(s.corner 280)]  node[below]  {{\small $j+1$}};
    \draw[shift=(s.corner 450)]  node[right]  {{\small $i$}};
   \end{tikzpicture}  
\end{center}
\caption{The local noncommutative exchange relations \\
$c_{j,i} = {\red c_{j,j+1}c_{i+1,j+1}^{-1}c_{i+1,i}} + {\blue c_{j,i+1} c_{j+1,i+1}^{-1}c_{j+1,i}}
$.}
\label{fig:localexchange}
\end{figure}

Let $R$ be a ring and 
$\mathcal{P}$ be an $m$-gon (where $m\ge 3$). Suppose that to each pair of vertices $i,j$ of
$\mathcal{P}$ invertible elements, values 
$c_{i,j}\in R^{\ast}$ and $c_{j,i}\in R^{\ast}$ are assigned, that is, we have a map 
$c:\mathrm{diag(\mathcal{P})}\to R^{\ast}$.

\begin{Lemma} \label{lem:localllocex}
If the values of the map $c$ satisfy all local triangle relations and all local noncommutative exchange
relations then they satisfy the triangle relations for all triangles with two consecutive vertices (i.e.\ triangles with
vertices $i,i+1,k$).
\end{Lemma}

\begin{proof}
We proceed by induction on $k$, where by symmetry we can assume that the vertex $k$ takes values
$i+2,i+3,\ldots$ (that is, we induct on the distance of $k$ to the edge with vertices $i,i+1$). 

If $k=i+2$, then the required triangle relation is a local triangle relation 
(cf.\ Definition \ref{def:nclocal}), and it is satisfied by assumption. 

For the inductive step we consider a triangle with vertices $i,i+1,k+1$. By inductive hypothesis we can assume
that the triangle relation holds for the triangle with vertices $i,i+1,k$. 
On the other hand, since the distance of the third vertex to the two consecutive ones is the same,
we can also assume by induction hypothesis that 
the triangle relation holds for the triangle with vertices $i+1,k,k+1$.

We need to show that 
$$ c_{i,k+1}c_{i+1,k+1}^{-1}c_{i+1,i} = c_{i,i+1}c_{k+1,i+1}^{-1}c_{k+1,i}.
$$
By assumption, the local noncommutative exchange relations are satisfied for the diagonals in the 
quadrangle with vertices $i,i+1,k,k+1$.  In particular we have
\begin{equation} \label{eq:locex1}
c_{i,k} = c_{i,i+1}c_{k+1,i+1}^{-1}c_{k+1,k} + c_{i,k+1}c_{i+1,k+1}^{-1} c_{i+1,k}
\end{equation}
and
\begin{equation} \label{eq:locex2}
c_{k,i} = c_{k,k+1}c_{i+1,k+1}^{-1}c_{i+1,i} + c_{k,i+1}c_{k+1,i+1}^{-1} c_{k+1,i}.
\end{equation}
Then we get
\begin{equation*}
c_{i,k+1}c_{i+1,k+1}^{-1}c_{i+1,i} \stackrel{(\ref{eq:locex1})}{=} 
c_{i,k}c_{i+1,k}^{-1}c_{i+1,i} - c_{i,i+1}c_{k+1,i+1}^{-1}c_{k+1,k}c_{i+1,k}^{-1}c_{i+1,i}. 
\end{equation*}
On the right hand side we apply the induction hypothesis (for the triangles with vertices $i,i+1,k$ and
$i,k,k+1$, respectively) and obtain
\begin{eqnarray*}
c_{i,k+1}c_{i+1,k+1}^{-1}c_{i+1,i} & = &  c_{i,i+1}c_{k,i+1}^{-1}c_{k,i} - 
c_{i,i+1}c_{k,i+1}^{-1} c_{k,k+1}c_{i+1,k+1}^{-1} c_{i+1,i} \\
& = & c_{i,i+1}(c_{k,i+1}^{-1}c_{k,i} - c_{k,i+1}^{-1} c_{k,k+1}c_{i+1,k+1}^{-1} c_{i+1,i}) \\
& \stackrel{(\ref{eq:locex2})}{=} & 
c_{i,i+1} c_{k+1,i+1}^{-1} c_{k+1,i}.
\end{eqnarray*}
This is the triangle relation for the triangle with vertices $i,i+1,k+1$, as required for the inductive step.
\end{proof}

\begin{Definition}
We call the triangle relations for triangles with two consecutive vertices {\em weak local} triangle relations.
\end{Definition}

\section{Noncommutative frieze patterns} \label{sec:friezepattern}
We transfer Berenstein-Retakh's notion of a noncommutative polygon (see Section \ref{sec:ncpr})
to the setting of frieze patterns. 
Recall that in the $m$-gon $\mathcal{P}$ we have for each pair of vertices $i,j$ two directed diagonals,
with corresponding values $c_{i,j}$ and $c_{j,i}$ (not necessarily equal as in the commutative case). 
These values can be arranged as usual into an array of numbers as in Figure \ref{fig:pattern}.
\begin{figure}
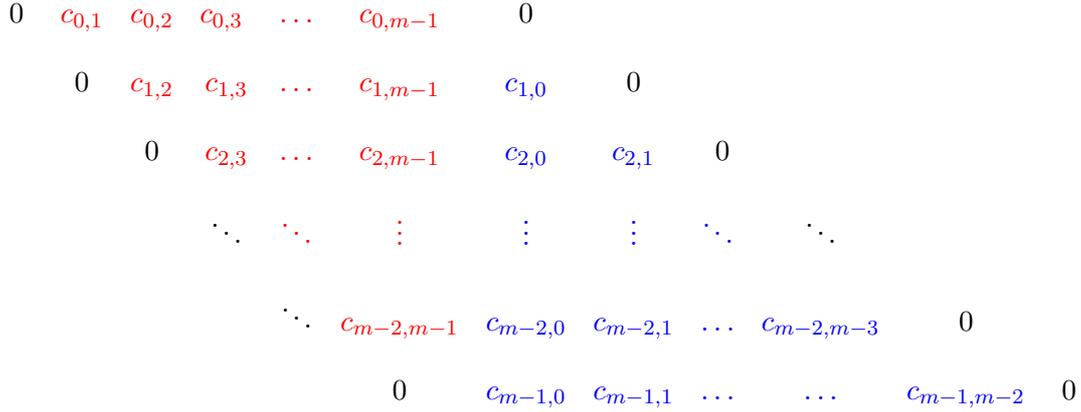
 
$$\begin{array}{ccccccccccccc}
0~~ & {\red c_{0,1}} & {\red c_{0,2}} & {\red c_{0,3}}~~~ & {\red \ldots} & {\red c_{0,m-1}} & 0 & & & & & &\\
~~ \\
& 0 & {\red c_{1,2}} & {\red c_{1,3}} & {\red \ldots} & {\red c_{1,m-1}} & {\blue c_{1,0}} & 0 & & & & &\\
~~ \\
& & 0 & {\red c_{2,3}} & {\red \ldots} & {\red c_{2,m-1}} & {\blue c_{2,0}} & {\blue c_{2,1}} & ~~~~~0 & & & &\\
~~\\
& &  & \ddots  &  {\red \ddots} & {\red \vdots} & {\blue \vdots} &  {\blue \vdots} & {\blue\ddots} & \ddots & & & \\
~~\\
& & &  & \ddots & {\red c_{m-2,m-1}} & {\blue c_{m-2,0}} & {\blue c_{m-2,1}} & {\blue \ldots}  &  {\blue c_{m-2,m-3}} & 0 & & \\
~~\\
& & & &  &   0 &  {\blue c_{m-1,0}} &  {\blue c_{m-1,1}} & {\blue \ldots} & {\blue \ldots} &  
{\blue c_{m-1,m-2}} & ~~~~~~~~0 &  \\
\end{array}
$$
\caption{A noncommutative frieze pattern}
 \label{fig:pattern}
\end{figure}
This is very reminiscent of a frieze pattern (with coefficients), as defined by Coxeter \cite{Cox71}
and later studied in numerous papers; see the survey \cite{MG15} for an overview and \cite{CHJ20}
for frieze patterns with coefficients. 
We refer to these classic frieze patterns as commutative
frieze patterns (with coefficients). 
In the commutative case, the blue region in Figure \ref{fig:pattern}
is obtained from the red region by glide reflection (since 
$c_{i,j}=c_{j,i}$), but this is no longer true in the noncommutative setting.

\begin{Definition} \label{def:ncfriezepattern}
Let $R$ be a ring. 
We call an array of invertible elements $c_{i,j}\in R^{\ast}$ as in Figure \ref{fig:pattern} 
a {\em noncommutative frieze pattern} of height $m-3$ over $R$ if all local triangle relations 
and all local noncommutative exchange relations (cf.\ Definition \ref{def:nclocal}) are satisfied.
\end{Definition}

We will sometimes have to distinguish between a noncommutative frieze pattern 
and a map $c:\mathrm{diag}(\mathcal{P})\to R^{\ast}$ on the diagonals of a polygon
which is called a noncommutative frieze (cf.\ Definition \ref{def:nclocal}).

\begin{Example} \label{ex:friezes}
\begin{enumerate}
\item[{(1)}] We consider noncommutative friezes on a triangle, with vertices $0,1,2$. The local
noncommutative exchange 
relations don't apply (since there are no quadrangles). The values on the edges of the triangle
 have to satisfy the triangle
relation 
\begin{equation} \label{eq:tritri}
c_{0,1} c_{2,1}^{-1} c_{2,0} = c_{0,2} c_{1,2}^{-1} c_{1,0}.
\end{equation}
The other triangle relations follow from this one, see Remark \ref{rem:relations}\,(2). 
Therefore, a noncommutative frieze on a triangle is given by values $c_{i,j}$ for $i,j\in \{0,1,2\}$ satisfying
the triangle relation (\ref{eq:tritri}). 
The corresponding noncommutative frieze pattern (of height 0) has the form
$$
\begin{array}{ccccccccc}
 & \ddots & &  & & & & & \\
0 & c_{0,1} & c_{0,2} & 0 & & & & & \\
& 0 & c_{1,2} & c_{1,0} & 0 & & & & \\
& & 0 & c_{2,0} & c_{2,1} & 0 & & & \\
& & & 0 & c_{0,1} & c_{0,2} & 0 & & \\
& & & & 0 & c_{1,2} & c_{1,0} & 0 &  \\
& & & & & 0 & c_{2,0} & c_{2,1} & 0  \\
& & & & &  & & \ddots &   \\
\end{array}
$$
where the entries have to satisfy the triangle relation (\ref{eq:tritri}).
As an explicit example of such a frieze pattern over the free skew field in two generators $x,y$ we can take
$c_{0,1}=x$, $c_{1,0}=x^{-1}$, $c_{0,2}=xy$, $c_{2,0}=y$, 
$c_{1,2}=y$, $c_{2,1}=yx$. Then the triangle relation reads 
$ x (yx)^{-1} y = xy y^{-1} x^{-1}$ which holds since both sides are equal to 1.
The corresponding noncommutative frieze pattern has the form
$$
\begin{array}{ccccccccc}
 & \ddots & &  & & & & & \\
0 & x & xy & 0 & & & & & \\
& 0 & y & x^{-1} & 0 & & & & \\
& & 0 & y & yx & 0 & & & \\
& & & 0 & x & xy & 0 & & \\
& & & & 0 & y & x^{-1} & 0 &  \\
& & & & & 0 & y & yx & 0  \\
& & & & &  & & \ddots &   \\
\end{array}
$$
Note that it is periodic of period 3 (because it comes from a frieze on a triangle), but there is no 
glide symmetry. 
\item[{(2)}] We give an example of a noncommutative frieze pattern of height 1 over 
the free skew field in two generators $x,y$ (the figure shows the fundamental domain which is repeated
by translation to both sides):
$$
\begin{array}{cccccccc}
 & & &  \ddots & & & & \\
0 & x & y & x & 0 & & &  \\
& 0 & x & 2xy^{-1}x & y & 0 & & \\
& & 0 & y & yx^{-1}yx^{-1}y & y & 0 & \\
& & & 0 & y & 2xy^{-1}x & x & 0  \\
& & & & & \ddots & &   \\
\end{array}
$$
\item[{(3)}] We consider the skew field of quaternions
$$\mathbb{H} = \{a+bi+cj+dk\,|\,a,b,c,d\in \mathbb{R}\}
$$
with the well-known multiplication given by the following table (and extended linearly):
$$\begin{array}{c|ccc}
 & i & j & k \\ \hline 
 i & -1 & k & -j \\
 j & -k & -1 & i \\
 k & j & -i & -1 
\end{array}
$$
Then we have the following example of a noncommutative frieze pattern over $\mathbb{H}$
of height 3 (i.e.\ it corresponds to a noncommutative
frieze on a hexagon):
$$
\begin{array}{cccccccccccc}
0 & ~1~ & ~~i~~ & 1 - k & -i - 2j & 1 & 0 & \sps & \sps & \sps & \sps & \sps \\
    & 0 & 1 & -2i - j & 3k & -i + j & 1 & 0 &    &    &    &   \\
    &    & 0 & 1 & i - j & k & i & 1 & 0 &    &    &   \\
    &    &    & 0 & 1 & j & 1 + k & -2i - j & 1 & 0 &    &   \\
    &    &    &    & 0 & 1 & -i - 2j & -3k & ~i - j~ & 1 & 0 &   \\
\sps & \sps & \sps & \sps & \sps & 0 & 1 & -i + j & -k & ~j~ & ~1~ & 0
\end{array}
$$
Note that this noncommutative frieze
has 1's on the boundary and that for all values on diagonals of length 3 we have $c_{i,i+3}\neq c_{i+3,i}$.
We leave it to the reader to check that the local triangle relations and the local noncommutative exchange
relations are satisfied.
\end{enumerate}
\end{Example}

\begin{Remark}
A noncommutative frieze (pattern) 
over a commutative set $R$ is not necessarily a classic frieze (pattern). 
For instance, let $R=\mathbb{N}$ and consider 
the frieze on a triangle with $c_{0,1}=c_{0,2}=2$ and $c_{1,2}=c_{2,1}=c_{1,0}=c_{2,0}=1$. This choice 
indeed yields a noncommutative frieze since the triangle relation 
$c_{0,1} c_{2,1}^{-1} c_{2,0} = c_{0,2} c_{1,2}^{-1} c_{1,0}$ from (\ref{eq:tritri}) is satisfied. 
However, it is not a classic frieze (with coefficients) since $c_{0,1}\neq c_{1,0}$.
\end{Remark}

\section{Quasideterminants} \label{sec:quasidet}
In the commutative setting, frieze patterns are defined by the local conditions that the determinant of
any neighbouring $2\times 2$-matrix is equal to 1 (or the product of two boundary entries for frieze patterns
with coefficients). In the noncommutative setting we will see that for 
defining noncommutative frieze patterns by local conditions the concept of {\em quasideterminants}
naturally appears. For $2\times 2$-matrices this goes back to Richardson \cite{R26} and Heyting
\cite{H28}, but the general theory and many fundamental properties of quasideterminants have been
established by Gelfand and Retakh \cite{GR91},\cite{GR92}.

We only include here the few definitions and results about quasideterminants which we shall
need later for our purposes. For more details we refer to the survey article
\cite{GGRW05}.

Let $A=(a_{i,j})$ be an $n\times n$-matrix. For any 
fixed $1\le i,j\le n$ let $A^{i,j}$ denote the matrix obtained from $A$ by deleting row $i$ and column $j$. 
The $(i,j)$-quasideterminant of $A$ is defined if the submatrix $A^{i,j}$ is invertible
and then it has the form
$$|A|_{i,j} = a_{i,j} - r_i^j (A^{i,j})^{-1} c_j^i
$$
where $r_i^j$ is the $i$th row of $A$ with $j$th entry deleted, and $c_j^i$ is the 
$j$th column of $A$ with entry $i$ deleted. 
\medskip

Of particular interest for us is the following characterisation for the vanishing of quasideterminants.

\begin{Proposition}(\cite[Proposition 1.4.6]{GGRW05}) \label{prop:qdet0}
Let $A=(a_{i,j})$ be a matrix over a division ring. If the quasideterminant $|A|_{i,j}$ is defined then the
following statements are equivalent.
\begin{enumerate}
\item[{(i)}] $|A|_{i,j}=0$.
\item[{(ii)}] The $i$th row of the matrix $A$ is a left linear combination of the other rows of $A$.
\item[{(iii)}] The $j$th column of the matrix $A$ is a right linear combination of the other
columns of $A$.
\end{enumerate}
\end{Proposition}

For later establishing fundamental properties of noncommutative frieze patterns 
we will need to consider neighbouring $3\times 3$-submatrices. For frieze patterns in the 
commutative case, the notion of tameness (introduced in \cite{BR10}) is fundamental, where a frieze pattern
is called tame if every neighbouring $3\times 3$-matrix in the frieze pattern has 
determinant 0. This is always the case when the entries in the frieze pattern are non-zero, see
\cite[Proposition 1]{BR10}.

In our noncommutative setting we will later need an analogous property when we 
want to prove propagation formulae. The determinants in the commutative case can be 
replaced by a certain quasideterminant. 
For this quasideterminant to be defined we shall need that in a noncommutative frieze pattern
every neighbouring $2\times 2$-matrix is invertible. 

We have the following general formula for the inverse of a $2\times 2$-matrix with
noncommuting entries which can be verified by a short and straightforward computation. 

\begin{Lemma} \label{lemma:inverse2by2}
For any invertible $2\times 2$-matrix $\begin{pmatrix} a & b \\ c & d \end{pmatrix}$ the inverse matrix is given 
by
$$\begin{pmatrix} -c^{-1}d(b-ac^{-1}d)^{-1} & -a^{-1}b(d-ca^{-1}b)^{-1} \\
(b-ac^{-1}d)^{-1} & (d-ca^{-1}b)^{-1}
\end{pmatrix}.
$$  
\end{Lemma}

For verifying that a certain $2\times 2$-matrix is invertible, it suffices to show that each entry in the
matrix in Lemma \ref{lemma:inverse2by2} is defined. 

For a noncommutative frieze pattern $\mathcal{C}=(c_{i,j})$ we consider a neighbouring $2\times 2$-submatrix
$$M':= \begin{pmatrix} c_{i,j-1} & c_{i,j} \\c_{i+1,j-1} & c_{i+1,j} \\
\end{pmatrix}.
$$
By Definition \ref{def:nclocal}, all entries in a noncommutative frieze pattern are invertible. The entries
of type $b-ac^{-1}d$ in Lemma \ref{lemma:inverse2by2} for the matrix $M'$ take the form
$c_{i,j}-c_{i,j-1}c_{i+1,j-1}^{-1}c_{i+1,j}$. Using the local noncommutative exchange relation for the
diagonal $(i,j-1)$ in the quadrangle with vertices $i,i+1,j-1,j$ we have
$$c_{i,j}-c_{i,j-1}c_{i+1,j-1}^{-1}c_{i+1,j} = - c_{i,i+1}c_{j,i+1}^{-1}c_{j,j-1}c_{i+1,j-1}^{-1}c_{i+1,j}
$$
and this is invertible (as each factor
is invertible by Definition \ref{def:nclocal}). Similarly, the entries
of type $d-ca^{-1}b$ for the matrix $M'$ are
$c_{i+1,j}-c_{i+1,j-1}c_{i,j-1}^{-1}c_{i,j}$ and this is equal
to $c_{i+1,i}c_{j-1,i}^{-1}c_{j-1,j}$ by the local noncommutative exchange relation for the 
diagonal $(i+1,j)$ in the quadrangle with vertices $i,i+1,j-1,j$. Again, this expression is invertible
by Definition \ref{def:nclocal}. So we have shown that every neighbouring $2\times 2$-submatrix
of a noncommutative frieze pattern is invertible. This implies that the quasideterminant appearing in the 
following definition exists.

\begin{Definition}
A noncommutative frieze pattern $\mathcal{C}=(c_{i,j})$ is called {\em tame} if
for every 
(neighbouring) $3\times 3$-submatrix 
$$M=\begin{pmatrix} c_{i,j-1} & c_{i,j} & c_{i,j+1} \\c_{i+1,j-1} & c_{i+1,j} & c_{i+1,j+1} \\
c_{i+2,j-1} & c_{i+2,j} & c_{i+2,j+1} \end{pmatrix} 
$$
the quasideterminant $|M|_{3,3}$ is equal to 0 (where indices are taken modulo $m$ for a frieze
pattern as in Figure \ref{fig:pattern}). 
\end{Definition}

We will see in Corollary \ref{prop:ncqdet0} that every noncommutative frieze pattern is tame.

\section{Theory of noncommutative frieze patterns} \label{sec:theoryncfriezes}

The aim of this section is to develop several fundamental properties of noncommutative 
frieze patterns. This generalizes results from our earlier paper \cite{CHJ20} on frieze 
patterns with coefficients from the commutative to the noncommutative setting. 

We first want to prove noncommutative versions of the propagation rules from \cite[Proposition 2.10]{CHJ20}.
For this we shall first need to deduce some results about relations in a noncommutative frieze pattern. 


\newcommand{\mlem}{i}
\newcommand{\jlem}{j}
\newcommand{\klem}{k}
\newcommand{\llem}{\ell}
\newcommand{\ilem}{m}
\newcommand{\T}[3]{c_{#2,#1}^{-1} c_{#2,#3} c_{#1,#3}^{-1}}
\newcommand{\yy}[3]{c_{#3,#1}^{-1} c_{#3,#2}}

\begin{Lemma}\label{lem_relations_step}
Let $\mathcal{C}=(c_{i,j})$ be a noncommutative frieze pattern.
Let $\mlem,\jlem ,\klem ,\llem ,\ilem $ be a cycle of indices such that the following relations hold:
\begin{enumerate}
\item[{(i)}] the noncommutative exchange relation 
for the diagonal $(\mlem,\klem)$ in the quadrangle $\klem ,\jlem ,\mlem,\llem$,
\item[{(ii)}] the noncommutative exchange relation for the diagonal $(\jlem,\ilem)$ in the quadrangle $\ilem ,\mlem,\jlem ,\llem $,
\item[{(iii)}] the triangle relations for the triangles $\klem ,\llem ,\ilem $ and $\mlem ,\jlem ,\llem$. 
\end{enumerate}
Then
\[
  \T{\ilem}{\jlem}{ \llem } = \T{\ilem}{\jlem}{ \klem } + \T{\ilem}{\klem}{ \llem } \quad \Longleftrightarrow \quad
  \T{\ilem}{\mlem}{ \llem } = \T{\ilem}{\mlem}{ \klem } + \T{\ilem}{\klem}{ \llem },
\]
i.e.\ the noncommutative exchange relations for the diagonal $(j,\ell)$ in the quadrangle $\llem,\klem,\jlem,\ilem$ 
and for the diagonal $(i,\ell)$ in the quadrangle $\llem,\klem,\mlem,\ilem$ are equivalent.
\end{Lemma}
\begin{proof}
The exchange relation $c_{\mlem,\llem} c_{\jlem,\llem}^{-1} c_{\jlem,\klem} = c_{\mlem,\klem} - c_{\mlem,\jlem} c_{\llem,\jlem}^{-1} c_{\llem,\klem}$ for $\klem ,\jlem ,\mlem,\llem $ gives
\begin{equation}\label{lem_eq_1}
\yy{\ilem}{\llem }{\mlem} \yy{\llem }{ \ilem }{\jlem}  \T{\ilem}{\jlem}{ \klem } = \T{\ilem}{\mlem}{ \klem } - \yy{\ilem}{ \jlem }{\mlem} \yy{\jlem}{ \ilem }{\llem}  \T{\ilem}{\llem}{  \klem }.
\end{equation}
Now assume $\T{\ilem}{\jlem}{ \llem } = \T{\ilem}{\jlem}{ \klem } + \T{\ilem}{\klem}{ \llem }$. Then
\begin{eqnarray*}
  \T{\ilem}{\mlem}{\llem } &=& \yy{\ilem}{ \llem }{\mlem} \yy{\llem}{\ilem }{\jlem}  \T{\ilem}{\jlem}{ \llem } \\
           &=& \yy{\ilem }{\llem }{\mlem }\yy{\llem }{ \ilem }{\jlem } (\T{\ilem}{\jlem}{ \klem } + \T{\ilem}{\klem }{\llem }) \\
           &=& \T{\ilem}{\mlem }{\klem } - \yy{\ilem}{ \jlem }{\mlem} \yy{\jlem}{\ilem }{\llem}  \T{\ilem}{\llem}{ \klem } + \yy{\ilem }{\llem }{\mlem} \yy{\llem  }{\ilem }{\jlem } \T{\ilem}{\llem }{ \klem }\\
           &=& \T{\ilem}{\mlem }{\klem } + \T{\ilem}{\klem}{\llem }
\end{eqnarray*}
where the third equality is by (\ref{lem_eq_1}) and the triangle $\klem ,\llem ,\ilem $, and the last equality holds
by the exchange relation $\ilem ,\mlem,\jlem ,\llem $ and the triangle $\mlem ,\jlem ,\llem$.
Thus in the sequence of equalities, the first and the last expressions are equal (exchange relation $\llem,\klem,\mlem,\ilem$) if and only if the second and the third expressions are equal (exchange relation $\llem,\klem,\jlem,\ilem$).
\end{proof}

As an application of the above lemma we can obtain new relations which have to be satisfied in any 
noncommutative frieze pattern.

\begin{Proposition}\label{exrel_3}
Let $\mathcal{C}=(c_{i,j})$ be a noncommutative frieze pattern.
Then the noncommutative exchange relations hold for the diagonals $(j,i+1)$ in every quadrangle
with vertices of the form $i+1,i+2,j,i$.
\end{Proposition}
\begin{proof}

We apply Lemma \ref{lem_relations_step} to the cycle $j+1,j,i+2,i+1,i$.
The assumptions are satisfied since the exchange relations $i+2,j,j+1,i+1$ and $i,j+1,j,i+1$ and the triangles $i,i+1,i+2$ and $i+1,j,j+1$ are (weak) local relations.
Thus the exchange relations $i+1,i+2,j,i$ and $i+1,i+2,j+1,i$ are equivalent.
The cycle $i+1,i+2,j,i$ for $j=i+3$ is a local relation; the claim follows for $j=i+4,\ldots,i-1$ by induction.
\end{proof}

\begin{Theorem} \label{prop:propagation}
Let $\mathcal{C}=(c_{i,j})$ be a noncommutative frieze pattern. 
\begin{enumerate}
\item[{(a)}] (Propagation along rows) For all $i,j$ we have
$$(c_{i,j-1},c_{i,j}) \begin{pmatrix} 0 & -c_{j,j-1}^{-1}c_{j,j+1} \\
1 & c_{j-1,j}^{-1}c_{j-1,j+1}\end{pmatrix} = (c_{i,j},c_{i,j+1}).
$$
\item[{(b)}] (Propagation along columns) For all $i,k$ we have
$$\begin{pmatrix} 0 & 1 \\ -c_{i+1,i}c_{i-1,i}^{-1} & c_{i+1,i-1}c_{i,i-1}^{-1}\end{pmatrix}
\begin{pmatrix} c_{i-1,k} \\ c_{i,k} \end{pmatrix} =
\begin{pmatrix} c_{i,k} \\ c_{i+1,k} \end{pmatrix}.
$$
\end{enumerate}
\end{Theorem}
\begin{proof} $(a)$
Since all local noncommutative exchange and local triangle relations hold, Proposition \ref{exrel_3} yields the relation for the cycle $j,j+1,i,j-1$ for all $i,j$:
\begin{equation}\label{thmeq_1}
c_{i,j} = c_{i,j-1} c_{j+1,j-1}^{-1} c_{j+1,j} + c_{i,j+1} c_{j-1,j+1}^{-1} c_{j-1,j}
\end{equation}
We write $s:=- c_{j,j-1}^{-1}c_{j,j+1}$, $t:=c_{j-1,j}^{-1}c_{j-1,j+1}$ for the two non-trivial entries in the matrix of the equation in $(a)$.
Multiplying Equation (\ref{thmeq_1}) by $t$ from the right gives
\begin{equation}\label{thmeq_2}
c_{i,j} t = c_{i,j-1} c_{j+1,j-1}^{-1} c_{j+1,j} c_{j-1,j}^{-1}c_{j-1,j+1} + c_{i,j+1}.
\end{equation}
The triangle relation $j-1,j,j+1$ gives
\[ s = -c_{j+1,j-1}^{-1} c_{j+1,j} c_{j-1,j}^{-1}c_{j-1,j+1} \]
thus Equation (\ref{thmeq_2}) becomes
\begin{equation}\label{thmeq_3}
c_{i,j+1} = c_{i,j-1} s + c_{i,j} t
\end{equation}
which proves $(a)$. For $(b)$, interchange the roles of rows and columns in the proof of $(a)$.
\end{proof}

\begin{Corollary} \label{prop:ncqdet0}
Let $\mathcal{C}=(c_{i,j})$ be a noncommutative frieze pattern. Then for every 
(neighbouring) $3\times 3$-submatrix 
$$M=\begin{pmatrix} c_{i,j} & c_{i,j+1} & c_{i,j+2} \\c_{i+1,j} & c_{i+1,j+1} & c_{i+1,j+2} \\
c_{i+2,j} & c_{i+2,j+1} & c_{i+2,j+2} \end{pmatrix} 
$$
the quasideterminant $|M|_{3,3}$ is equal to 0.
In particular, every noncommutative frieze pattern is tame.
\end{Corollary}

\begin{proof}
We have observed at the end of Section \ref{sec:quasidet} that the quasideterminant $|M|_{3,3}$ 
is defined.
By Theorem \ref{prop:propagation} $(a)$, the third column is a 
right linear combination of the first two columns, thus the quasideterminant $|M|_{3,3}$ is equal to $0$ by Proposition \ref{prop:qdet0}. 
\end{proof}

\begin{Definition} (Noncommutative $\mu$-matrices) \label{def:mumatrix}
For any four numbers $c,d,e,f$ we define a $2\times 2$-matrix
$$\mu(c,d,e,f) = \begin{pmatrix} 0 & -e^{-1}d \\ 1 & f^{-1}c \end{pmatrix}.
$$

\end{Definition}

\begin{Remark} \label{rem:propagation}
\begin{enumerate}
\item[{(a)}] With the notation from Definition \ref{def:mumatrix} the propagation rules in
Proposition \ref{prop:propagation} can be written as

\noindent
(Propagation along rows) 
$$(c_{i,j-1},c_{i,j}) \mu(c_{j-1,j+1},c_{j,j+1},c_{j,j-1},c_{j-1,j}) = (c_{i,j},c_{i,j+1}).
$$
(Propagation along columns) 
$$\mu(c_{i,i-1}^{-1}, c_{i-1,i}^{-1},c_{i+1,i}^{-1},c_{i+1,i-1}^{-1})^T
\begin{pmatrix} c_{i-1,k} \\ c_{i,k} \end{pmatrix} =
\begin{pmatrix} c_{i,k} \\ c_{i+1,k} \end{pmatrix}.
$$
where the exponent $T$ denotes transposition of matrices.
\item[{(b)}] 
In the commutative setting the $\mu$-matrices for the propagation rules had the form
(cf.\ \cite[Definition 2.8]{CHJ20}):
$$ \begin{pmatrix} 0 & -\frac{c_{j,j+1}}{c_{j-1,j}} \\ 1 & \frac{c_{j-1,j+1}}{c_{j-1,j}} \end{pmatrix}.
$$
Comparing this to the matrices in Definition \ref{def:mumatrix} and Remark \ref{rem:propagation}
note that the main difference occurs for the denominators in the second column which are
equal in the commutative setting but are in general different for noncommutative frieze patterns.
\end{enumerate}
\end{Remark}

\begin{Remark} (Extended noncommutative frieze patterns) \label{rem:extended}
Let $\mathcal{C}=(c_{i,j})$ be a noncommutative frieze pattern (as in Section \ref{sec:friezepattern}). 
The corresponding {\em extended noncommutative frieze pattern} $\hat{\mathcal{C}}$ is the 
infinite array of the form
$${\footnotesize
\begin{array}{cccccccccc}
& & &  \ddots & & & &  & & \\
{\blue -c_{i-1,i-2}} & 0 & c_{i-1,i} & c_{i-1,i+1} & \ldots & c_{i-1,i+m-2} & 0 & -c_{i-1,i} & & \\
& {\blue -c_{i,i-1}} & 0 & c_{i,i+1} & c_{i,i+2} & \ldots & c_{i,i+m-1} & 0 & -c_{i,i+1} &  \\
& & {\blue -c_{i+1,i}} & 0 & c_{i+1,i+2} & c_{i+1,i+3} & \ldots & c_{i+1,i+m} 
& 0 & -c_{i+1,i+2} \\
& & & & & & \ddots & & & 
\end{array}
}
$$
Compared to the commutative setting as in \cite[Remark 2.7]{CHJ20}, only the diagonal highlighted in blue
is changed, where now $-c_{i,i-1}$ appears (instead of $-c_{i-1,i}$ in the commutative case). 
Note that with this convention, the propagation formula in Proposition \ref{prop:propagation} also holds
for the boundary cases $j=i$ and $j=i+m$ in part (a) and similarly for part (b).  
\end{Remark}

\begin{Corollary} (Consequences of the propagation rules) \label{cor:consprop}
Let $\mathcal{C}=(c_{i,j})$ be a noncommutative frieze pattern.
\begin{enumerate}
\item[{(a)}] For all $i,j$ we have
\begin{eqnarray*}
M_{i,j} & := & \prod_{r=i}^j \mu(c_{r-1,r+1},c_{r,r+1},c_{r,r-1},c_{r-1,r}) \\
& = & \begin{pmatrix} -c_{i,i-1} & 0 \\ 0 & c_{i-1,i}\end{pmatrix}^{-1} 
\begin{pmatrix} c_{i,j} & c_{i,j+1} \\ c_{i-1,j} & c_{i-1,j+1} \end{pmatrix} \\
& = & \begin{pmatrix} -c_{i,i-1}^{-1} c_{i,j} & -c_{i,i-1}^{-1} c_{i,j+1} \\
c_{i-1,i}^{-1} c_{i-1,j} & c_{i-1,i}^{-1}c_{i-1,j+1} \end{pmatrix}.
\end{eqnarray*}
\item[{(b)}] We have
$$\prod_{r=1}^m \mu(c_{r-1,r+1},c_{r,r+1},c_{r,r-1},c_{r-1,r}) = \begin{pmatrix} -1 & 0 \\ 0 & -1 
\end{pmatrix} =:-\mathrm{id}.
$$
\end{enumerate}
\end{Corollary}

\begin{proof}
(a) Using the extended frieze pattern from Remark \ref{rem:extended}, the propagation formulae
in Proposition \ref{prop:propagation} yield
$$({\blue -c_{i,i-1}},\underbrace{c_{i,i}}_{=0}) M_{i,j} = (c_{i,j},c_{i,j+1})
$$
and 
$$(\underbrace{c_{i-1,i-1}}_{=0},c_{i-1,i}) M_{i,j} = (c_{i-1,j},c_{i-1,j+1}).
$$
Taken together, these equation yield the assertion in part (a).
\smallskip

\noindent
(b) This follows from part (a) since in the extended frieze pattern we have $\hat{c}_{0,m+1}
=-c_{0,1}$. 
\end{proof}

We want to generalize the classic notion of a quiddity cycle, due to Conway and Coxeter, to
noncommutative frieze patterns (with coefficients). Our approach is based on the formula
in Corollary \ref{cor:consprop}\,(b) (see also Definition \ref{def:mumatrix}).

\begin{Definition} \label{def:quiddity}
Let $m\in \mathbb{N}$. A sequence 
$(d_0,d_0',d_1,d_1',\ldots,d_{m-1},d_{m-1}';c_0,c_1,\ldots,c_{m-1})$ of elements of a ring $R$ is called a 
{\em noncommutative quiddity cycle} if $d_i$, $d_i'$ are invertible for all $i$ and 
$$\mu(c_0,d_1,d_0',d_0)\mu(c_1,d_2,d_1',d_1)\ldots \mu(c_{m-2},d_{m-1},d_{m-2}',d_{m-2})
\mu(c_{m-1},d_{0},d_{m-1}',d_{m-1}) = -\mathrm{id}.
$$
\end{Definition}

\begin{Remark}
In the classic Conway-Coxeter theory all boundary entries are equal to 1 and the classic  
quiddity cycles appear as the special case 
$(1,1,1,1,\ldots,1,1;c_0,c_1,\ldots,c_{m-1})$
of Definition~\ref{def:quiddity}.
\end{Remark}

\begin{Example}
\begin{enumerate}
\item[{(1)}] For $m=1$ there does not exist any noncommutative quiddity cycle since
$\mu(c_0,d_0,d_0',d_0) \neq -\mathrm{id}$ by Definition \ref{def:mumatrix}.
\item[{(2)}] For $m=2$ we have
\begin{eqnarray*}
\mu(c_0,d_1,d_0',d_0)\mu(c_1,d_0,d_1',d_1) & = & 
\begin{pmatrix} 0 & -(d_0')^{-1}d_1 \\ 1 & d_0^{-1}c_0 \end{pmatrix}
\begin{pmatrix} 0 & -(d_1')^{-1}d_0 \\ 1 & d_1^{-1}c_1 \end{pmatrix} \\
& = & \begin{pmatrix} -(d_0')^{-1}d_1 & -(d_0')^{-1}c_1 \\ 
d_0^{-1}c_0 & -(d_1')^{-1}d_0 + d_0^{-1}c_0d_1^{-1}c_1 \end{pmatrix}.
\end{eqnarray*}
This matrix becomes $-\mathrm{id}$ if and only if $c_0=0=c_1$ and 
$d_0=d_1'$ and $d_1=d_0'$. Hence, the noncommutative quiddity cycles for $m=2$ are
$(d_0,d_0',d_0',d_0;0,0)$ with invertible elements $d_0$, $d_0'$.
\item[{(3)}] For $m=3$ we have that 
$$\mu(c_0,d_1,d_0',d_0)\mu(c_1,d_2,d_1',d_1)\mu(c_2,d_0,d_2',d_2)$$ is equal to
$$
\begin{pmatrix} -(d_0')^{-1}d_1 & -(d_0')^{-1}c_1 \\ 
d_0^{-1}c_0 & -(d_1')^{-1}d_2 + d_0^{-1}c_0d_1^{-1}c_1 \end{pmatrix} 
\begin{pmatrix} 0 & -(d_2')^{-1}d_0 \\ 1 & d_2^{-1}c_2\end{pmatrix}
$$
which is then computed to be equal to
$$\begin{pmatrix}
-(d_0')^{-1}c_1 & (d_0')^{-1}d_1(d_2')^{-1}d_0 - (d_0')^{-1}c_1 d_2^{-1}c_2 \\
-(d_1')^{-1}d_2 + d_0^{-1}c_0d_1^{-1}c_1 & -d_0^{-1}c_0(d_2')^{-1}d_0 - (d_1')^{-1}c_2
+ d_0^{-1}c_0d_1^{-1}c_1d_2^{-1}c_2 
\end{pmatrix}.
$$
We want to determine when this matrix becomes $-\mathrm{id}$. From the $(1,1)$-entry we get
\begin{equation} \label{eq:c1}
c_1=d_0'.
\end{equation}
We can use this equation and solve the off-diagonal entries (which need to be $0$) 
for $c_0$ and $c_2$, respectively. Then 
we obtain the conditions
\begin{equation} \label{eq:c0}
c_0 = d_0(d_1')^{-1}d_2 (d_0')^{-1}d_1
\end{equation}
and
\begin{equation} \label{eq:c2}
c_2 = d_2(d_0')^{-1}d_1 (d_2')^{-1}d_0.
\end{equation}
The remaining $(2,2)$-entry  
$$-d_0^{-1}c_0(d_2')^{-1}d_0 - (d_1')^{-1}c_2+ d_0^{-1}c_0d_1^{-1}c_1d_2^{-1}c_2
$$ 
has to be equal to $-1$. Plugging in equations (\ref{eq:c0}), (\ref{eq:c2}) and (\ref{eq:c1}) one 
observes that (after cancellations) all three summands become equal up to signs and hence the
$(2,2)$-entry is equal to 
$$-(d_1')^{-1}d_2(d_0')^{-1}d_1(d_2')^{-1}d_0.
$$
For a noncommutative quiddity cycle we thus get the condition 
\begin{equation} \label{eq:qctriangle}
d_2(d_0')^{-1}d_1 = d_1' d_0^{-1} d_2'.
\end{equation}
Note that this is a triangle relation (see Definition \ref{eq:triangle}) where $d_i,d_i'$ denote the
values on arrows in opposite directions. 
Inserting this triangle relation into equations (\ref{eq:c0}) and (\ref{eq:c2}) we obtain 
$c_0=d_2'$ and $c_2=d_1'$, respectively. 

We summarize the above calculations: all noncommutative quiddity cycles for $m=3$ are given by
$$\{ (d_0,d_0',d_1,d_1',d_2,d_2'; d_2',d_0',d_1')\,|\, d_2(d_0')^{-1}d_1 = d_1' d_0^{-1} d_2'\}.
$$

As a special case for the classic situation (commutative, all boundary values equal to 1) we get back
the well-known result that there is only one quiddity cycle for $m=3$, with all entries equal to 1. 
\end{enumerate}
\end{Example}

The following result shows that any noncommutative quiddity cycle (with $m\ge 3$)
can be reduced to a shorter noncommutative quiddity cycle.
Our motivation for stating this formula comes from noncommutative frieze patterns, but this is
a general formula for the $\mu$-matrices as in Definition \ref{def:mumatrix}.

\begin{Proposition} \label{prop:qcreduce}
Consider elements $c_{i-1},c_i,c_{i+1}$ and invertible elements $d_{i-1}$, $d_i$, $d_{i+1}$, $d'_{i-1}$, $d'_i$, $d'_{i+1}$ in a ring $R$.
Then the following products of matrices are equal:
\begin{enumerate}
\item[{(i)}] $\mu(c_{i-1},d_i,d_{i-1}',d_{i-1})\mu(c_{i},d_{i+1},d_{i}',d_{i})
\mu(c_{i+1},d_{i+2},d_{i+1}',d_{i+1})$
\item[{(ii)}] $\mu(-d_{i-1}(d_i')^{-1}d_{i+1}+c_{i-1}d_i^{-1}c_i,c_i,d_{i-1}',d_{i-1}) \:\:\cdot$ \\
$\mu(-d_i(d_{i+1}')^{-1}d_{i+2}+c_i d_{i+1}^{-1}c_{i+1},d_{i+2}, d_{i+1}'d_i^{-1}c_i d_{i+1}^{-1}d_i',c_i)$
\end{enumerate}
\end{Proposition}

\begin{proof}
This can be verified by a straightforward (though slightly tedious) computation.
\end{proof}

\begin{Remark}
\begin{enumerate}
\item[{(1)}] We consider the special case of Proposition \ref{prop:qcreduce} where all boundary values
are equal to 1. Then the formula reads
$$\mu(c_{i-1},1,1,1)\mu(c_{i},1,1,1)
\mu(c_{i+1},1,1,1) = 
\mu(-1+c_{i-1}c_i,c_i,1,1)
\mu(-1+c_ic_{i+1},1,c_i,c_i).
$$
We see that some boundary entries become $c_i$ after the reduction, that is, we do not get a classic
frieze any more. However, when $c_i=1$ we get back the well-known reduction formula
$$\mu(c_{i-1},1,1,1)\mu(1,1,1,1)
\mu(c_{i+1},1,1,1) = 
\mu(-1+c_{i-1},1,1,1)
\mu(-1+c_{i+1},1,1,1)
$$
(inserting/deleting a 1 in the quiddity cycle) which underlies the classic bijection between 
frieze patterns over natural numbers and triangulations of polygons by Conway and Coxeter.
\item[{(2)}] Suppose we have a noncommutative frieze pattern (on some $m$-gon) 
$\mathcal{C}=(c_{i,j})$. 
Then Corollary \ref{cor:consprop}\,(b) states that the product of certain $\mu$-matrices is the negative
of the identity matrix. Let us consider a subproduct of three of these $\mu$-matrices, that is,
$$\mu(c_{r-2,r},c_{r-1,r},c_{r-1,r-2},c_{r-2,r-1})\mu(c_{r-1,r+1},c_{r,r+1},c_{r,r-1},c_{r-1,r})
\mu(c_{r,r+2},c_{r+1,r+2},c_{r+1,r},c_{r,r+1}).
$$
This corresponds to the part of the noncommutative polygon as in Figure \ref{fig:5vertices}.
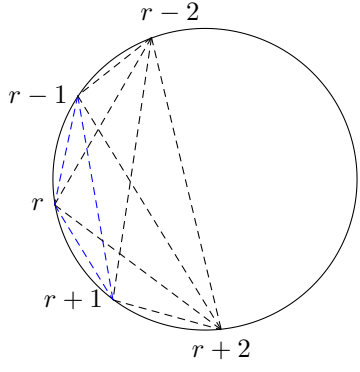
\begin{figure}
  \begin{tikzpicture}[auto]
    \node[name=s, draw, shape=regular polygon, regular polygon sides=500, minimum size=4cm] {};
   
    \draw[densely dashed] (s.corner 30) to (s.corner 80);
    \draw[densely dashed] (s.corner 30) to (s.corner 140);
    \draw[densely dashed] (s.corner 30) to (s.corner 200);
    \draw[densely dashed] (s.corner 30) to (s.corner 260);
     \draw[blue,densely dashed] (s.corner 80) to (s.corner 140);
     \draw[blue,densely dashed] (s.corner 80) to (s.corner 200);
     \draw[densely dashed] (s.corner 80) to (s.corner 260);
     \draw[blue,densely dashed] (s.corner 140) to (s.corner 200);
     \draw[densely dashed] (s.corner 140) to (s.corner 260);
     \draw[densely dashed] (s.corner 200) to (s.corner 260);

     \draw[shift=(s.corner 20)]  node[above]  {{\small $r-2$}};
    \draw[shift=(s.corner 80)]  node[left]  {{\small $r-1$}};
    \draw[shift=(s.corner 140)]  node[left]  {{\small $r$}};
    \draw[shift=(s.corner 200)]  node[left]  {{\small $r+1$}};
 \draw[shift=(s.corner 260)]  node[below]  {{\small $r+2$}};
 
   \end{tikzpicture}
   \caption{A subpolygon as part of a noncommutative quiddity cycle. \label{fig:5vertices}}
\end{figure}
The formula in Proposition \ref{prop:qcreduce} states that the above triple product of $\mu$-matrices 
is equal to the double product
\begin{equation} \label{eq:doubleprod}
\mu({\red -c_{r-2,r-1}c_{r,r-1}^{-1}c_{r,r+1}+c_{r-2,r}c_{r-1,r}^{-1}c_{r-1,r+1}}, 
c_{r-1,r+1},c_{r-1,r-2}, c_{r-2,r-1}) \cdot
\end{equation}
$$ \mu({\red -c_{r-1,r}c_{r+1,r}^{-1}c_{r+1,r+2}+c_{r-1,r+1}c_{r,r+1}^{-1}c_{r,r+2}}, 
c_{r+1,r+2},{\blue c_{r+1,r}c_{r-1,r}^{-1}c_{r-1,r+1}}c_{r,r+1}^{-1}c_{r,r-1}, c_{r-1,r+1}).
$$
The entries are part of a noncommutative frieze pattern, in particular we can apply the triangle relation
to the product highlighted in blue (corresponding to the blue triangle in Figure \ref{fig:5vertices}). 
Then we get 
$${\blue c_{r+1,r}c_{r-1,r}^{-1}c_{r-1,r+1}}c_{r,r+1}^{-1}c_{r,r-1}
= c_{r+1,r-1}c_{r,r-1}^{-1}c_{r,r+1}c_{r,r+1}^{-1}c_{r,r-1} = c_{r+1,r-1}.
$$
Moreover, the expressions highlighted in red in equation (\ref{eq:doubleprod}) can be shown to be equal
to $c_{r-2,r+1}$ and $c_{r-1,r+2}$, respectively (by using 
the noncommutative exchange relations for the quadrangles with vertices $r-2,r-1,r,r+1$
and $r-1,r,r+1,r+2$, respectively, and some triangle relations). 
Altogether, we obtain that the product in equation (\ref{eq:doubleprod})
has the form
$$\mu(c_{r-2,r+1},c_{r-1,r+1},c_{r-1,r-2}, c_{r-2,r-1}) \mu(c_{r-1,r+2},c_{r+1,r+2},c_{r+1,r-1},c_{r-1,r+1}).
$$
Note that this is precisely the part of the product as in Corollary \ref{cor:consprop} for the 
noncommutative polygon as in Figure \ref{fig:4vertices}, that is,
where the blue triangle in Figure \ref{fig:5vertices} has been cut off. 
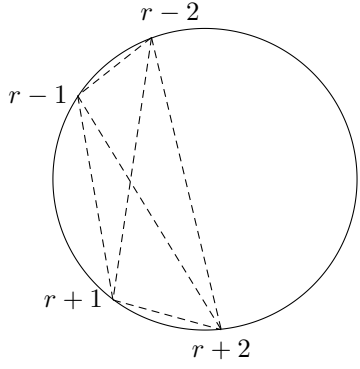
\begin{figure}
  \begin{tikzpicture}[auto]
    \node[name=s, draw, shape=regular polygon, regular polygon sides=500, minimum size=4cm] {};
   
    \draw[densely dashed] (s.corner 30) to (s.corner 80);
    \draw[densely dashed] (s.corner 30) to (s.corner 200);
    \draw[densely dashed] (s.corner 30) to (s.corner 260);
     \draw[densely dashed] (s.corner 80) to (s.corner 200);
     \draw[densely dashed] (s.corner 80) to (s.corner 260);
     \draw[densely dashed] (s.corner 200) to (s.corner 260);

     \draw[shift=(s.corner 20)]  node[above]  {{\small $r-2$}};
    \draw[shift=(s.corner 80)]  node[left]  {{\small $r-1$}};
    \draw[shift=(s.corner 200)]  node[left]  {{\small $r+1$}};
 \draw[shift=(s.corner 260)]  node[below]  {{\small $r+2$}};
 
   \end{tikzpicture}
   \caption{The subpolygon after reducing a noncommutative quiddity cycle. \label{fig:4vertices}}
\end{figure}
This shows that the reduction rule in Proposition \ref{prop:qcreduce} corresponds to the geometric
construction of cutting one vertex in a noncommutative polygon. 
\end{enumerate}
\end{Remark}

\section{Local relations imply all relations for noncommutative friezes}
\label{sec:localimplyall}

Our main result in this section is the following. 

\begin{Theorem} \label{thm:localimplyall}
Let $c:\mathrm{diag}(\mathcal{P})\to R^{\ast}$ be a noncommutative frieze on a polygon $\mathcal{P}$
over a ring $R$. Then $c$ satisfies all triangle relations and all noncommutative exchange relations.
\end{Theorem}

The proof of Theorem \ref{thm:localimplyall} will consist of several steps provided by a series of 
lemmas. From now on, we fix the following assumptions. Let $R$ be a ring and 
$\mathcal{P}$ be an $m$-gon (where $m\ge 3$). Suppose that to each pair of vertices $i,j$ of
$\mathcal{P}$ invertible elements, values 
$c_{i,j}\in R^{\ast}$ and $c_{j,i}\in R^{\ast}$ are assigned, that is, we have a map 
$c:\mathrm{diag(\mathcal{P})}\to R^{\ast}$.

\begin{Lemma} \label{lem:weakexchange}
If $c$ satisfies all weak local triangle relations and all local noncommutative exchange relations
then it also satisfies the noncommutative exchange relations in quadrangles of the form $(i,j-1,j,k)$ for
the diagonals $(i,j)$ and $(j-1,k)$.  
\end{Lemma}

\begin{proof}
We consider four vertices $i<j-1<j<k$. We know from Corollary \ref{cor:consprop} that
for every $r,s$ we have
$$
M_{r,s} = \begin{pmatrix} -c_{r,r-1}^{-1} c_{r,s} & -c_{r,r-1}^{-1} c_{r,s+1} \\
c_{r-1,r}^{-1} c_{r-1,s} & c_{r-1,r}^{-1}c_{r-1,s+1} \end{pmatrix}.
$$

It is important to note here that in the proofs of Corollary \ref{cor:consprop} and Theorem 
\ref{prop:propagation} and Proposition \ref{exrel_3}
on which it builds, only local noncommutative exchange relations and weak local
triangle relations have been used. This means that with the assumptions of the lemma which we are about to
prove it is indeed allowed to use Corollary \ref{cor:consprop}.

We have $M_{i,k-1} = M_{i,j-1} M_{j,k-1}$ (by definition of these matrices). Comparing the top right
entries gives
\begin{equation} \label{eq:ij-1jk}
-c_{i,i-1}^{-1} c_{i,k} = c_{i,i-1}^{-1}c_{i,j-1}c_{j,j-1}^{-1} c_{j,k} - 
c_{i,i-1}^{-1}c_{i,j}c_{j-1,j}^{-1}c_{j-1,k}.
\end{equation}
Upon multiplication with $-c_{i,j-1}^{-1}c_{i,i-1}$ from the left and $c_{j-1,k}^{-1}$ from the right this becomes 
$$c_{i,j-1}^{-1}c_{i,k}c_{j-1,k}^{-1} = - c_{j,j-1}^{-1}c_{j,k}c_{j-1,k}^{-1} + c_{i,j-1}^{-1}c_{i,j}c_{j-1,j}^{-1}.
$$
We apply a weak local triangle relation to first summand on the right and multiply by $c_{i,j-1}$ from the left
and $c_{j-1,j}$ from the right to obtain
\begin{equation} \label{eq:weakexchangecij}
c_{i,k}c_{j-1,k}^{-1}c_{j-1,j} = - c_{i,j-1}c_{k,j-1}^{-1}c_{k,j} + c_{i,j}.
\end{equation}
This is the noncommutative exchange relation in the quadrangle $(i,j-1,j,k)$ for the diagonal $(i,j)$, as claimed. 

For proving the noncommutative exchange relation for the diagonal $(j-1,k)$ we multiply equation
(\ref{eq:ij-1jk}) with $-c_{j-1,j}c_{i,j}^{-1}c_{i,i-1}$ and obtain
$$c_{j-1,j}c_{i,j}^{-1}c_{i,k} = - c_{j-1,j}c_{i,j}^{-1}c_{i,j-1}c_{j,j-1}^{-1}c_{j,k} + c_{j-1,k}.
$$
We apply a weak local triangle relation to the first factors of the first summand on the right and get
$$c_{j-1,j} c_{i,j}^{-1}c_{i,k}  + c_{j-1,i}c_{j,i}^{-1}c_{j,k} = c_{j-1,k}
$$
and this is the desired noncommutative exchange relation for the diagonal $(j-1,k)$.
\end{proof}

As a main step towards proving Theorem \ref{thm:localimplyall} we can now show how to obtain
all triangle relations. 

\begin{Lemma} \label{lem:alltriangles}
If $c$ satisfies all local triangle relations and all local noncommutative exchange relations
then it satisfies all triangle relations. 
\end{Lemma}

\begin{proof}
From Lemma \ref{lem:localllocex} we know that the map $c$ satisfies all weak local triangle relations. 
Then Lemma \ref{lem:weakexchange} implies that $c$ satisfies certain non-local exchange relations,
namely in quadrangles with two consecutive vertices for two of the directed diagonals.   

Following the notation as in Lemma \ref{lem:weakexchange}, we consider an arbitrary triangle with
vertices $i,j-1,k$. Again we use the matrices $M_{r,s}$ and Corollary \ref{cor:consprop}. We consider 
the following products (the equalities follow from the definition of the matrices $M_{r,s}$ and
Corollary \ref{cor:consprop}):
$$-\mathrm{id}\cdot M_{k,i-1} = M_{k,k-1}\cdot M_{k,i-1} = M_{k,j-1}\cdot M_{j,k-1}\cdot M_{k,i-1}
= M_{k,j-1}\cdot M_{j,i-1}.
$$
Comparing the top right entries of the matrices on the left and the right we get
$$c_{k,k-1}^{-1} c_{k,i} = c_{k,k-1}^{-1}c_{k,j-1}c_{j,j-1}^{-1}c_{j,i} - 
c_{k,k-1}^{-1}c_{k,j}c_{j-1,j}^{-1} c_{j-1,i}.
$$
When we multiply this equation from the left with $c_{k,j-1}^{-1}c_{k,k-1}$ and from the right with $c_{j-1,i}^{-1}$
we obtain
$$c_{k,j-1}^{-1} c_{k,i}c_{j-1,i}^{-1} = c_{j,j-1}^{-1}c_{j,i}c_{j-1,i}^{-1} - c_{k,j-1}^{-1}c_{k,j}c_{j-1,j}^{-1}.
$$
Using a weak local triangle relation for the first summand on the right yields
$$c_{k,j-1}^{-1} c_{k,i}c_{j-1,i}^{-1} = c_{i,j-1}^{-1}c_{i,j}c_{j-1,j}^{-1} - c_{k,j-1}^{-1}c_{k,j}c_{j-1,j}^{-1}.
$$
We can now insert the noncommutative exchange relation for $c_{i,j}$ from Lemma \ref{lem:weakexchange},
i.e.\ we use equation (\ref{eq:weakexchangecij}) and get
\begin{eqnarray*}
c_{k,j-1}^{-1} c_{k,i}c_{j-1,i}^{-1} & = & c_{i,j-1}^{-1} (c_{i,j-1}c_{k,j-1}^{-1}c_{k,j}
 + c_{i,k}c_{j-1,k}^{-1}c_{j-1,j})c_{j-1,j}^{-1} - c_{k,j-1}^{-1}c_{k,j}c_{j-1,j}^{-1} \\
 & = & c_{k,j-1}^{-1}c_{k,j}c_{j-1,j}^{-1} + c_{i,j-1}^{-1}c_{i,k}c_{j-1,k}^{-1}  
 - c_{k,j-1}^{-1}c_{k,j}c_{j-1,j}^{-1} \\
 & = & c_{i,j-1}^{-1}c_{i,k}c_{j-1,k}^{-1}. 
\end{eqnarray*}
Taking inverses we get
$$c_{j-1,i} c_{k,i}^{-1} c_{k,j-1} = c_{j-1,k} c_{i,k}^{-1}c_{i,j-1}
$$
and this is a triangle relation for the (arbitrary) triangle $(i,j-1,k)$. Recall that by 
Remark \ref{rem:relations}\,(2) it suffices to show one of the triangle relations for any given triangle,
so the proof is complete. 
\end{proof}

The final step for proving Theorem \ref{thm:localimplyall} is to deduce all noncommutative exchange 
relations. 

\begin{Lemma} \label{lem:allexchange}
If $c$ satisfies all local triangle relations and all local noncommutative exchange relations
then it satisfies all noncommutative exchange relations. 
\end{Lemma}

\begin{proof}
We consider an arbitrary quadrangle, with vertices $i<j<k<\ell$. 
We have to show that the noncommutative exchange relations hold for this quadrangle. 
Note that there are four different relations, one for each directed diagonal of the quadrangle, but that by rotational 
symmetry it suffices to show the noncommutative exchange relation for only one of these 
diagonals.  We will show the noncommutative exchange relation for the diagonal $(i,k)$ inductively. To this end,
we fix the vertices $i,j,\ell$ and do induction on $k=j+1,j+2,\ldots$. As a base for the induction we consider the 
quadrangle $(i,j,j+1,\ell)$. The noncommutative exchange relation for the diagonal $(i,k)=(i,j+1)$ 
holds by Lemma \ref{lem:weakexchange}. Now suppose as induction hypothesis that the 
noncommutative exchange relation holds for the diagonal $(i,k-1)$, i.e.
\begin{equation} \label{eq:ih}
c_{i,k-1} = c_{i,j}c_{\ell,j}^{-1}c_{\ell,k-1}+c_{i,\ell}c_{j,\ell}^{-1}c_{j,k-1}.
\end{equation} 

We start with the noncommutative exchange relation for the diagonal $(i,k)$ in the quadrangle 
$(i,k-1,k,\ell)$, which holds by Lemma \ref{lem:weakexchange}, i.e.\ we have
$$c_{i,k} = c_{i,k-1}c_{\ell,k-1}^{-1}c_{\ell,k} + c_{i,\ell}c_{k-1,\ell}^{-1}c_{k-1,k}.
$$
Inserting the induction hypothesis (\ref{eq:ih}) we obtain
\begin{eqnarray*}
c_{i,k} & = & (c_{i,j}c_{\ell,j}^{-1}c_{\ell,k-1}+c_{i,\ell}c_{j,\ell}^{-1}c_{j,k-1})
 c_{\ell,k-1}^{-1}c_{\ell,k} + c_{i,\ell}c_{k-1,\ell}^{-1}c_{k-1,k} \\
 & = &c_{i,j}c_{\ell,j}^{-1}c_{\ell,k} + c_{i,\ell}c_{j,\ell}^{-1}c_{j,k-1}c_{\ell,k-1}^{-1}c_{\ell,k}
 + c_{i,\ell}c_{k-1,\ell}^{-1}c_{k-1,k}.
 \end{eqnarray*}
From Lemma \ref{lem:alltriangles} we know that the map $c$ satisfies all triangle relations. We use the triangle 
relation for the triangle $(j,k-1,\ell)$, i.e.\ 
$c_{j,\ell}c_{k-1,\ell}^{-1}c_{k-1,j}=c_{j,k-1}c_{\ell,k-1}^{-1}c_{\ell,j}$, 
 and insert it for $c_{k-1,\ell}^{-1}$ in the above equation:
\begin{eqnarray*}
c_{i,k} & = & c_{i,j}c_{\ell,j}^{-1}c_{\ell,k} + c_{i,\ell}c_{j,\ell}^{-1}c_{j,k-1}c_{\ell,k-1}^{-1}c_{\ell,k}
+ c_{i,\ell}c_{j,\ell}^{-1} c_{j,k-1}c_{\ell,k-1}^{-1}c_{\ell,j}c_{k-1,j}^{-1}c_{k-1,k} \\
& = & c_{i,j} c_{\ell,j}^{-1}c_{\ell,k} + c_{i,\ell}c_{j,\ell}^{-1}c_{j,k-1}c_{\ell,k-1}^{-1}
(c_{\ell,k} + c_{\ell,j}c_{k-1,j}^{-1}c_{k-1,k}).
\end{eqnarray*}
By Lemma \ref{lem:weakexchange} the noncommutative exchange relation for the diagonal $(j,k)$ in the
quadrangle $(j,k-1,k,\ell)$ holds; by elementary transformations this relation has the form
$$c_{\ell,k} = c_{\ell,k-1} c_{j,k-1}^{-1}c_{j,k} - c_{\ell,j}c_{k-1,j}^{-1}c_{k-1,k}.
$$ 
Inserting this into the above equation we get
$$c_{i,k} = c_{i,j}c_{\ell,j}^{-1}c_{\ell,k} + c_{i,\ell}c_{j,\ell}^{-1}c_{j,k-1}c_{\ell,k-1}^{-1}
c_{\ell,k-1} c_{j,k-1}^{-1}c_{j,k} = c_{i,j}c_{\ell,j}^{-1}c_{\ell,k} + c_{i,\ell} c_{j,\ell}^{-1}c_{j,k}.
$$
This is the noncommutative exchange relation for the diagonal $(i,k)$ in the quadrangle $(i,j,k,\ell)$
which had to be shown to complete the proof of the lemma.
\end{proof}

Finally, combining Lemmas 
\ref{lem:alltriangles} and \ref{lem:allexchange}
yields a proof of our main result Theorem \ref{thm:localimplyall}.


\end{document}